\documentclass[12pt,reqno]{amsart}
\usepackage{amsmath, amsfonts, amssymb, amsthm}
\def\N{\mathbb N}

\def\cI{\mathcal I}

\def\cA{\mathcal A}

\def\fJ{{\mathfrak J}}

\def\1{{\bf 1}}

\theoremstyle{plain}
\newtheorem{theorem}{Theorem}

\newtheorem{lemma}{Lemma}
\newtheorem{corollary}{Corollary}
\theoremstyle{definition}

\theoremstyle{remark}

\begin{document}

\title{On the Erd\H os-Fuchs theorem}
\author{Li-Xia Dai}
\email{lilidainjnu@163.com}
\address{School of Mathematical Sciences, Nanjing Normal University, Nanjing 210046, People's Republic of China}
\author{Hao Pan}
\email{haopan79@zoho.com}
\address{Department of Mathematics, Nanjing University,
Nanjing 210093, People's Republic of China}
\thanks{}
\begin{abstract}
We prove several extensions of the Erd\H os-Fuchs theorem.
\end{abstract}
\maketitle

\section{Introduction}
\setcounter{lemma}{0}\setcounter{theorem}{0}\setcounter{proposition}{0}\setcounter{corollary}{0}
\setcounter{equation}{0}

The well-known Gauss circle conjecture says that
\begin{equation}\label{Gauss}
|\{(a,b):\,a,b\in\N,a^2+b^2\leq n\}|=\frac{\pi}{4}n+O(n^{\frac14+\epsilon})
\end{equation}
for any $\epsilon>0$. The known best result due to Huxley is replacing $O(n^{\frac14+\epsilon})$
by $O(n^{\frac{131}{416}})$. In general,
for two non-empty subsets $A,B$ of $\N$ and $n\in\N$, define
$$
r_{A,B}(n):=|\{(a,b):\,a\in A,\ b\in B,\ a+b=n\}|.
$$
Also, define
$$
R_{A,B}(n):=\sum_{j\leq n}r_{A,B}(j),
$$
i.e.,
$$
R_{A,B}(n)=|\{(a,b):\,a\in A,\ b\in B,\ a+b\leq n\}|.
$$
Clearly (\ref{Gauss}) can be rewritten as
$$
R_{\N^2,\N^2}(n)=\frac\pi4 n+O(n^{\frac14+\epsilon}),
$$
where $\N^2=\{a^2:\,a\in\N\}$.

On the other hand, with help of the Fourier analysis, Hardy
found that the remainder $O(n^{\frac14+\epsilon})$ in (\ref{Gauss}) can't be replaced by $O\big(n^\frac14(\log n)^{\frac14}\big)$.
In 1956, for arbitrary non-empty infinite subset $A$ of $\N$, Erd\H os and Fuchs
\cite{EF56} proved that as $n\to+\infty$,
\begin{equation}\label{EF}
R_{A,A}(n)=cn+o\big(n^{\frac14}(\log n)^{-\frac12}\big)
\end{equation}
can't hold for any constant $c>0$. This result is so-called the {\it Erd\H os-Fuchs theorem}. 
Subsequently, Jurkat (unpublished), and later Montgomery and Vaughan \cite{MV90} showed that the $(\log n)^{-\frac12}$ in the remainder term of (\ref{EF})
can be removed, i.e., it is impossible that
\begin{equation}\label{JMV}
R_{A,A}(n)=cn+o(n^{\frac14}),\qquad n\to\infty
\end{equation}
for some constant $c>0$.

In \cite{Sa80}, Sark\"ozy considered the extension of the Er\H os-Fuchs theorem for the sum of are different subsets of $\N$.
Let $A=\{a_1,a_2,\ldots\}$ and $B=\{b_1,b_2,\ldots\}$ be two infinite subsets of $\N$.
Suppose that for each $i\geq 1$, $a_i$ is not very far from $b_i$, explicitly,
\begin{equation}
a_i-b_i=o\big(a_i^\frac12(\log a_i)^{-1}\big).\tag{S}
\end{equation}
Then Sark\"ozy proved that 
\begin{equation}\label{Sarkozy}
R_{A,B}(n)=cn+o\big(n^{\frac14}(\log n)^{-\frac12}\big),\qquad n\to+\infty
\end{equation}
can not hold for any constant $c>0$.

In \cite{Ho04}, Horv\'ath 
tried to remove the term $(\log n)^{-\frac12}$ in the right side of (\ref{Sarkozy}). Define
$$
A(n):=|\{a\in A:\,a\leq n\}|.
$$
Under two assumptions
\begin{align}
&a_i-b_i=o(a_i^\frac12),\qquad i\to+\infty,\tag{H1}\\
&A(n)-B(n)=O(1),\qquad n\geq 1,\tag{H2}\label{AnBnO1}
\end{align}
Horv\'ath proved that
\begin{equation}\label{Hor}
R_{A,B}(n)=cn+o(n^{\frac14}),\qquad n\to+\infty
\end{equation}
would not happen.

Notice that the assumption (\ref{AnBnO1}), which says $A(n)$ and $B(n)$ are almost equal, seems a little too strong.
So we wish to weaken the requirement for $A(n)-B(n)$, under the assumption that the difference $a_i-b_i$ is much smaller than $o(a_i^\frac12)$.
In this paper, we shall give such a generalization of  Horv\'ath's result.
\begin{theorem}\label{mainT1}
Suppose that $0\leq\alpha\leq1/4$. Let
$A=\{a_1,a_2,\ldots\}$ and $B=\{b_1,b_2,\ldots\}$ be two infinite subsets of $\N$ satisfying that\medskip

\noindent(1) $a_i-b_i=o(a_i^{\frac12-\alpha})$
as $i\to+\infty$;\medskip

\noindent(2) $A(n)-B(n)=O(n^{\alpha})$
for each $n\in\N$. \medskip

\noindent Then
\begin{equation}\label{EFAB}
R_{A,B}(n)=cn+o(n^{\frac{1}{4}}),\qquad n\to\infty
\end{equation}
can not hold for any constant $c>0$.
\end{theorem}
Note that $a_i-b_i=o(a_i^{\frac12-\alpha})$ implies $A(n)-B(n)=o(n^{\frac12-\alpha})$.
Hence setting $\alpha=1/4$, we get
\begin{corollary}
$$
R_{A,B}(n)=cn+o(n^{\frac{1}{4}}),\qquad n\to\infty
$$
can't hold for any constant $c>0$, under the unique assumption
$$
a_i-b_i=o(a_i^{\frac14}).
$$
\end{corollary}
We also can consider the generalizations of the Erd\H os-Fuchs theorem for the sums of more than two subsets of $\N$. Suppose that $A_1,A_2,\ldots,A_k$ are non-empty subsets of $\N$. Define
$$
r_{A_1,\ldots,A_k}(n)=|\{(a_1,a_2,\ldots,a_k):\,a_1+\cdots+a_k=n,\ a_1\in A_1,\ldots,a_k\in A_k\}|,
$$
and define
$$
R_{A_1,\ldots,A_k}(n)=\sum_{j\leq n}r_{A_1,\ldots,A_k}(j).
$$
Horv\'ath \cite{Ho01,Ho02} proved that for any $A\subseteq\N$,
\begin{equation}
\underbrace{R_{A,\ldots,A}}_{k\ A\text{'s}}(n)=cn+o(n^{\frac14}(\log n)^{-\frac12})
\end{equation}
can't hold for any constant $c>0$. Subsequently, Tang \cite{Ta09} obtained an extension of (\ref{JMV}) for the sum of $k$ $A$'s, i.e., it is impossible that
\begin{equation}\label{TangkA}
\underbrace{R_{A,\ldots,A}}_{k\ \text{times}}(n)=cn+o(n^{\frac14}).
\end{equation}
Chen and Tang also proved a quantitative version of (\ref{TangkA}) in \cite{CT11}.

In \cite{Ho01,Ho02}, Horv\'ath factly considered $R_{A_1,\ldots,A_k}(n)$. Assume that
$A_1=\{a_{1,1},a_{1,2},\ldots\}$ and $A_2=\{a_{1,1},a_{1,2},\ldots\}$. Suppose that
\begin{align}
&a_{1,i}-a_{2,i}=o\big(a_{1,i}^{\frac12}(\log a_{1,i})^{-\frac k2}\big),\qquad i\to+\infty,\tag{h1}\\
&\label{AjnA1n}A_j(n)=\Theta\big(A_1(n)\big),\qquad j=3,\ldots,k\tag{h2}
\end{align}
for any sufficiently large $n$,
where $f=\Theta(g)$ means $g\ll f\ll g$, i.e., there exist two constants $c_1,c_2>0$ such that $c_1g(n)\leq f(n)\leq c_2g(n)$ for any sufficiently large $n$. Then
Horv\'ath \cite{Ho02} showed that for any constant $c>0$,
\begin{equation}
R_{A_1,\ldots,A_k}(n)=cn+o(n^{\frac14}(\log n)^{1-\frac{3k}4}),\qquad n\to+\infty,
\end{equation}
is impossible. Under some additional assumptions, Tang \cite{Ta14} improved  Horv\'ath's result and showed that the remainder term can be reduced to $o(n^{\frac14}(\log n)^{-\frac{1}2})$ or $o(n^{\frac14}(\log n)^{-\frac{k+1}{2(k-1)}})$ according to whether $k$ is even or odd.

Here we shall give an extension of (\ref{JMV}) concerning $R_{A_1,\ldots,A_k}(n)$.
\begin{theorem}\label{mainT2}
Suppose that $0<\beta\leq 1/2$ and $0\leq\alpha\leq\beta/2$. Let
$A_1,A_2,\ldots,A_k$ be some non-empty subsets of $\N$ satisfying that\medskip

\noindent(1) $a_{1,i}-a_{2,i}=o(a_{1,i}^{\beta-\alpha})$
as $i\to+\infty$, where $a_{1,i}$ (resp. $a_{2,i}$) is the $i$-th elements of $A_1$ (resp. $A_2$);\medskip

\noindent(2) $A_1(n)-A_2(n)=O(n^{\alpha})$
for each $n\in\N$; \medskip

\noindent(3) $A_1(n)=\Theta(n^\beta)$. \medskip

\noindent Then
\begin{equation}\label{EFAk}
R_{A_1,\ldots,A_k}(n)=cn+o(n^{\frac{1}{4}}),\qquad n\to\infty
\end{equation}
can not hold for any constant $c>0$.
\end{theorem}
Clearly the assumptions (2) and (3) of Theorem \ref{mainT2} also imply $A_2(n)=\Theta(n^\beta)$. Furthermore, if $A_j(n)=\Theta\big(A_1(n)\big)$ for $j=2,\ldots,k$ and $R_{A_1,\ldots,A_k}(n)=\Theta(n)$, then it is easy to verify that 
$A_1(n)=\Theta(n^\frac{1}{k})$.
Hence (2) and (3) of Theorem \ref{mainT2} are valid under Horv\'ath's assumption (\ref{AjnA1n}).

In \cite{Ba77}, Bateman showed that
\begin{equation}\label{Bateman}
\sum_{j\leq n}\big(R_{A,A}(j)-cn\big)^2=o\big(n^{\frac32}(\log n)^{-1}\big),\qquad n\to+\infty
\end{equation} 
can't hold for any constant $c>0$. Clearly the result of Bateman implies the Erd\H os-Fuchs theorem.
In \cite{CT12}, Chen and Tang showed that for any constant $c>0$, it is impossible that
\begin{equation}\label{CT}
\sum_{j\leq n}\big(\underbrace{R_{A,\ldots,A}}_{k\ \text{times}}(j)-cn\big)^2=o\big(n^{\frac32}\big),\qquad n\to+\infty.
\end{equation} 

Now we can prove that
\begin{theorem}\label{mainT3}
(i) Under the assumptions of Theorem \ref{mainT1}, 
\begin{equation}\label{BatemanAB}
\sum_{j\leq n}\big(R_{A,B}(j)-cn\big)^2=o(n^{\frac{3}{2}}),\qquad n\to\infty
\end{equation}
can not hold for any constant $c>0$.\medskip

\noindent(ii) Under the assumptions of Theorem \ref{mainT2},
\begin{equation}\label{CTA1Ak}
\sum_{j\leq n}\big(R_{A_1,\ldots,A_k}(j)-cn\big)^2=o(n^{\frac{3}{2}}),\qquad n\to\infty
\end{equation}
can not hold for any constant $c>0$.
\end{theorem}
In Section 2, we shall establish an auxiliary lemma and use it to conclude the proof of Theorem \ref{mainT1}. This lemma is also necessary to the proof of Theorem \ref{mainT2}, which will be given in Section 3. Finally, we shall prove Theorem \ref{mainT3} in Section 4.

\section{An auxiliary lemma and the proof of Theorem \ref{mainT1}}
\setcounter{lemma}{0}\setcounter{theorem}{0}\setcounter{proposition}{0}\setcounter{corollary}{0}
\setcounter{equation}{0}

\begin{lemma}\label{NABdiffT}
Suppose that $0<\beta\leq 1/2$ and $0\leq\alpha\leq\beta/2$. Let
$A=\{a_1,a_2,\ldots\}$ and $B=\{b_1,b_2,\ldots\}$ be two infinite subsets of $\N$ satisfying that\medskip

\noindent(1) $a_i-b_i=o(a_i^{\beta-\alpha})$
for each $i\geq 1$;\medskip

\noindent(2) $A(n)-B(n)=O(n^{\alpha})$
for each $n\in\N$; \medskip

\noindent(3) $A(n),B(n)\leq cn^\beta$
for each $n\in\N$, where $c>0$ is a constant. \medskip

\noindent Then as $N\to+\infty$, we have
\begin{equation}\label{NABabdiff}
\sum_{n=0}^\infty \bigg(1-\frac1N\bigg)^{2n}\cdot\bigg(\sum_{\substack{a\in A\\ a\leq n}}a-\sum_{\substack{b\in B\\ b\leq n}}b\bigg)^2=o(N^{2+2\beta}),
\end{equation}
and
\begin{equation}\label{NABdiff}
\sum_{n=0}^\infty \bigg(1-\frac1N\bigg)^{2n}\cdot\big(A(n)-B(n)\big)^2=o(N^{2\beta}).
\end{equation}
\end{lemma}
\begin{proof}
For each $j\geq 1$, let the interval
$$
\cI_j=[\min\{a_{j},b_{j}\},\max\{a_{j},b_{j}\}-1].
$$
Evidently $|\cI_j|=o({a_{j}}^{\beta-\alpha})$.
Define
$$
\lambda(n)=|\{j:\, n\in\cI_j\}|.
$$
Note that $n\in\cI_j$ if and only if either $a_j\leq n<b_j$ or 
$b_j\leq n<a_j$.
Hence
$$
\lambda(n)=|A(n)-B(n)|=O(n^\alpha).
$$
Thus
\begin{align*}
\bigg|\sum_{\substack{a\in A\\ a\leq n}}a-\sum_{\substack{b\in B\\ b\leq n}}b\bigg|
\leq&\sum_{\substack{j\geq 1\\ n\in\cI_j}}\min\{a_j,b_j\}+\sum_{\substack{j\geq 1\\ a_j,b_j\leq n}}|a_j-b_j|\\
\leq&\lambda(n)\cdot n+A(n)\cdot o(n^{\beta-\alpha})
=\lambda(n)\cdot n+o(n^{2\beta-\alpha}),
\end{align*}
where in the last step we used the assumption $A(n)\ll n^\beta$.
It follows that
\begin{align*}
\sum_{n\leq x}\bigg(\sum_{\substack{a\in A\\ a\leq n}}a-\sum_{\substack{b\in B\\ b\leq n}}b\bigg)^2\ll\sum_{n\leq x}\big(\lambda(n)^2n^2+o(n^{4\beta-2\alpha})\big)\ll
x^{2}\sum_{n\leq x}\lambda(n)^2+o(x^{1+4\beta-2\alpha})
\end{align*}
for any sufficiently large $x$.
Define
$$
\fJ(x)=\max\{j\geq 1:\,a_j\leq x\text{ or }b_j\leq x\}.
$$
Clearly $\fJ(x)\ll x^\beta$. We have
\begin{align*}
&\sum_{n\leq x}\lambda(n)^2\ll
x^\alpha\sum_{n\leq x}\lambda(n)\leq 
x^\alpha\sum_{j\leq \fJ(x)}|\cI_j|=x^\alpha\cdot x^\beta\cdot o(x^{\beta-\alpha})=o(x^{2\beta}),
\end{align*}
as $x\to+\infty$, i.e.,
$$
\sum_{n\leq x}\bigg(\sum_{\substack{a\in A\\ a\leq n}}a-\sum_{\substack{b\in B\\ b\leq n}}b\bigg)^2=o(x^{2+2\beta}).
$$

Now for any $\epsilon>0$, there exists $x_0=x_0(\epsilon)>0$ such that if $x\geq x_0$, then
$$
\sum_{n\leq x}\bigg(\sum_{\substack{a\in A\\ a\leq n}}a-\sum_{\substack{b\in B\\ b\leq n}}b\bigg)^2\leq\epsilon x^{2+2\beta}.
$$
Define
$$
\psi(x)=\sum_{n\leq x}\bigg(\sum_{\substack{a\in A\\ a\leq n}}a-\sum_{\substack{b\in B\\ b\leq n}}b\bigg)^2.
$$
Trivially, $\psi(x)=\sum_{n<x}(n\cdot n)^2\leq x^5$
for any $x\geq 0$.
Let $\rho=1-1/N$. 
Applying the Stieltjes integral, we get
\begin{align*}
\sum_{n=0}^\infty\rho^{2n}\bigg(\sum_{\substack{a\in A\\ a\leq n}}a-\sum_{\substack{b\in B\\ b\leq n}}b\bigg)^2
=&\int_0^\infty\rho^{2x} d\psi(x)\\
=&
\rho^{2x}\psi(x)\big|_{x=0}^{+\infty}-2\int_0^{+\infty}\psi(x)\cdot\rho^{2x}\log\rho dx.
\end{align*}
Since $0<\rho<1$,
$$
\lim_{x\to+\infty}\rho^{2x}\psi(x)\leq \lim_{x\to+\infty}\rho^{2x}x^{2+2\beta}=0.
$$
Let $\eta=-\log\rho$. Clearly $1/N\leq \eta\leq 2/N$. Then
\begin{align*}
&-\int_0^{+\infty}\psi(x)\cdot\rho^{2x}\log\rho dx=
\eta\int_0^{+\infty}\psi(x)\cdot e^{-2\eta x} dx\\
=& 
\eta\int_{x_0}^{+\infty}\psi(x)\cdot e^{-2\eta x} dx+
\eta\int_{0}^{x_0}\psi(x)\cdot e^{-2\eta x} dx\\
\leq&
\eta\int_{0}^{+\infty}\epsilon x^{2+2\beta}\cdot e^{-2\eta x} dx+
\eta \int_{0}^{x_0} x^5dx\\
=&\frac{\Gamma(3+2\beta)\cdot\epsilon}{2^{3+2\beta}\eta^{2+2\beta}}+\frac{\eta x_0^6}{6}\leq
\frac{\cdot\Gamma(3+2\beta)\cdot \epsilon N^{2+2\beta}}{2^{3+2\beta}}+\frac{x_0^6}{3N},
\end{align*}
where $\Gamma$ is the Gamma function.
If $N\geq \epsilon^{-1}x_0^6$, we can get
$$
\sum_{n=0}^\infty\rho^{2n}\bigg(\sum_{\substack{a\in A\\ a\leq n}}a-\sum_{\substack{b\in B\\ b\leq n}}b\bigg)^2
\leq 4\epsilon N^{2+2\beta}.
$$
Since $\epsilon>0$ can be arbitrarily small, we get (\ref{NABabdiff}).

Similarly, we have
$$
\sum_{n\leq x}\big(A(n)-B(n)\big)^2=
\sum_{n\leq x}\lambda(n)^2\ll x^{\alpha}\sum_{n\leq x}\lambda(n)\leq 
x^\alpha\sum_{j\leq \fJ(x)}|\cI_j|=o(x^{2\beta}),
$$
as $x\to+\infty$. For any $\epsilon>0$, there exists $x_0=x_0(\epsilon)>0$ such that for any $x>x_0$,
$$
\sum_{n\leq x}(A(n)-B(n))^2\leq \epsilon x^{2\beta}.
$$
Define
$$
\phi(x)=\sum_{n\leq x}(A(n)-B(n))^2.
$$
We also have
\begin{align*}
\sum_{n=0}^\infty\rho^{2n}\big(A(n)-B(n)\big)^2
=&
\rho^{2x}\phi(x)\big|_{x=0}^{+\infty}-2\int_0^{+\infty}\phi(x)\cdot\rho^{2x}\log\rho dx\\
=& 
2\eta\int_{x_0}^{+\infty}\phi(x)\cdot e^{-2\eta x} dx+
2\eta\int_{0}^{x_0}\phi(x)\cdot e^{-2\eta x} dx\\
\leq&
2\eta\int_{0}^{+\infty}\epsilon x^{2\beta}\cdot e^{-2\eta x} dx+
2\eta \int_{0}^{x_0} x^3dx\\
\leq&\frac{\Gamma(1+2\beta)\cdot \epsilon N^{2\beta}}{2^{2\beta}}+\frac{x_0^4}{N}\leq 4\epsilon N^{\beta}
\end{align*}
provided that $N\geq\epsilon^{-1}x_0^4$.
Thus  (\ref{NABdiff}) is concluded, too.
\end{proof}
Now we are ready to prove Theorem \ref{mainT1}.
\begin{proof}[Proof of Theorem \ref{mainT1}]
Assume on the contrary that (\ref{EFAB}) is true.
Define 
$\vartheta(n)=R_{A,B}(n)-cn$.  Then
$\vartheta(n)=o(n^{\frac14})$.
Furthermore, since
$$
A(n)B(n)\leq R_{A,B}(2n)=2cn+o(n^{\frac14}),
$$
we also have
$A(n),B(n)\leq 2c^\frac12n^\frac12$
for the sufficiently large $n$.

For $|z|<1$, let
$$
F(z)=\sum_{a\in A}z^a,\qquad G(z)=\sum_{b\in B}z^b.
$$
Clearly
$$
F(z)G(z)=\bigg(\sum_{a\in A}z^a\bigg)\cdot\bigg(\sum_{b\in B}z^b\bigg)=\sum_{n=0}^\infty z^n\sum_{\substack{a\in A,\ b\in B\\ a+b=n}}1=
\sum_{n=0}^\infty r_{A,B}(n)z^n.
$$
It follows that
$$
\frac{F(z)G(z)}{1-z}=\bigg(\sum_{n=0}^\infty r_{A,B}(n)z^n\bigg)\cdot\bigg(\sum_{n=0}^\infty z^n\bigg)
\sum_{n=0}^\infty z^n\sum_{j=0}^nr_{A,B}(j)=\sum_{n=0}^\infty R_{A,B}(n)z^n.
$$
Thus
\begin{align*}
\frac{(F(z)+G(z))^2}{4(1-z)}=&\frac{F(z)G(z)}{1-z}+\frac{(F(z)-G(z))^2}{4(1-z)}\\=
&c\sum_{n=0}^\infty nz^n+\sum_{n=0}^n\vartheta(n)z^n+\frac{(F(z)-G(z))^2}{4(1-z)},
\end{align*}
i.e.,
\begin{align*}
\frac{(F(z)+G(z))^2}{2}=\frac{2cz}{1-z}+2(1-z)\sum_{n=0}^n\vartheta(n)z^n+\frac{(F(z)-G(z))^2}{2}.
\end{align*}
Taking the derivative in $z$ of both sides of the above equation, we get
\begin{align}\label{FGFG}
(F'(z)+G'(z))(F(z)+G(z))=&\frac{2c}{(1-z)^2}+(F'(z)-G'(z))(F(z)-G(z))\notag\\
&+2(1-z)\sum_{n=1}^nn\vartheta(n)z^{n-1}-2\sum_{n=0}^n\vartheta(n)z^n.
\end{align}

Let $m$ be a large integer to be chosen later. 
Let $\rho=1-1/N$ and $z(\theta)=\rho e^{2\pi\sqrt{-1}\theta}$.
For convenience, we abbreviate $z(\theta)$ as $z$. Clearly for any $n_1,n_2\in\N$,
$$
\int_0^1 z^{n_1}\cdot \overline{z}^{n_2}d\theta=\begin{cases}\rho^{2n_1},&\text{if }n_1=n_2,\\
0,&\text{otherwise}.
\end{cases}
$$
Let
\begin{equation}\label{J}
J=\int_0^1 \big|(F'(z)+G'(z))(F(z)+G(z))\big|\cdot\bigg|\frac{1-z^m}{1-z} \bigg|^2d\theta,
\end{equation}
\begin{equation}\label{J1}
J_1=\int_0^1 \bigg| \frac{2c}{(1-z)^2} \bigg|\cdot\bigg|\frac{1-z^m}{1-z} \bigg|^2d\theta,
\end{equation}
\begin{equation}\label{J2}
J_2=\int_0^1 \bigg|2\sum_{n=0}^\infty \vartheta(n)z^n \bigg|\cdot\bigg|\frac{1-z^m}{1-z} \bigg|^2d\theta,
\end{equation}
\begin{equation}\label{J3}
J_3=\int_0^1 \bigg|2(1-z)\sum_{n=0}^\infty (n+1)v(n+1)z^n \bigg|\cdot\bigg|\frac{1-z^m}{1-z} \bigg|^2d\theta,
\end{equation}
and
\begin{equation}\label{J4}
J_4=\int_0^1 \big|\big(F'(z)-G'(z) \big)\big(F(z)-G(z) \big)\big|\cdot\bigg|\frac{1-z^m}{1-z} \bigg|^2d\theta.
\end{equation}
 
Evidently by (\ref{FGFG}), we have
$$
J\leq J_1+J_2+J_3+J_4.
$$
In \cite{Ho04}, Horv\'ath showed that
$$
J\gg m N^{\frac32},\qquad J_1,J_2\ll m^2N,\qquad J_3=o(m^{\frac12}N^{\frac 74}).
$$ 
We only need to give an upper bound for $J_4$. By the Cauchy-Schwarz inequality,
\begin{align*}
J_4\leq &4\int_0^1 \bigg|\frac{F'(z)-G'(z)}{1-z}\bigg|\cdot\bigg|\frac{F(z)-G(z)}{1-z}\bigg|d\theta\\
\leq&\frac{4}{\rho}\bigg(\int_0^1 \bigg|\frac{zF'(z)-zG'(z)}{1-z}\bigg|^2d\theta\bigg)^\frac12\bigg(\int_0^1 \bigg|\frac{F(z)-G(z)}{1-z}\bigg|^2d\theta\bigg)^\frac12.
 \end{align*}
Note that
$$
\frac{zF'(z)-zG'(z)}{1-z}=\frac1{1-z}\bigg(\sum_{a\in A}az^a-\sum_{b\in B}bz^b\bigg)=
\sum_{n=0}^\infty z^n\bigg(\sum_{\substack{a\in A\\ a\leq n}}a-\sum_{\substack{b\in B\\ b\leq n}}b\bigg).
$$
Applying (\ref{NABabdiff}) with $\beta=1/2$, we have
\begin{align*}
&\int_0^1 \bigg|\frac{zF'(z)-zG'(z)}{1-z}\bigg|^2d\theta\\
=&\int_0^1\bigg(\sum_{n=0}^\infty z^n\bigg(\sum_{\substack{a\in A\\ a\leq n}}a-\sum_{\substack{b\in B\\ b\leq n}}b\bigg)\bigg)\bigg(\sum_{n=0}^\infty \bar{z}^n\bigg(\sum_{\substack{a\in A\\ a\leq n}}a-\sum_{\substack{b\in B\\ b\leq n}}b\bigg)\bigg)d\theta
\\=&\sum_{n=0}^\infty\rho^{2n}\bigg(\sum_{\substack{a\in A\\ a\leq n}}a-\sum_{\substack{b\in B\\ b\leq n}}b\bigg)^2=o(N^{3}).
\end{align*}
Similarly, by (\ref{NABdiff}),
\begin{align*}
\int_0^1 \bigg|\frac{F(z)-G(z)}{1-z}\bigg|^2d\theta
=\sum_{n=0}^\infty\rho^{2n}\bigg(\sum_{\substack{a\in A\\ a\leq n}}1-\sum_{\substack{b\in B\\ b\leq n}}1\bigg)^2=o(N).
\end{align*}
Thus $J_4=o(N^2)$.

Since $J\leq J_1+J_2+J_3+J_4$, there exists a constant $C>1$ such that
$$
m N^{\frac32}\leq Cm^2N+o(m^{\frac12}N^{\frac 74})+o(N^2).
$$
By letting $m=C^{-2}N^{\frac12}$, we get an evident contradiction.
\end{proof}

\section{Proof of Theorem \ref{mainT2}}
\setcounter{lemma}{0}\setcounter{theorem}{0}\setcounter{proposition}{0}\setcounter{corollary}{0}
\setcounter{equation}{0}

Let us turn to Theorem \ref{mainT2}. 
\begin{lemma}\label{A3AkL}
Suppose that $0<\beta\leq 1/2$ and $A_1,\ldots,A_k$ are non-empty subsets of $\N$.
Assume that
$A_1(n),A_2(n)=\Theta(n^\beta)$ and $R_{A_1,\ldots,A_k}(n)=\Theta(n)$. Then
$$
R_{A_3,\ldots,A_k}(n)=\Theta(n^{1-2\beta}).
$$
\end{lemma}
\begin{proof} Evidently
\begin{align*}
R_{A_1,\ldots,A_k}(n)=&\sum_{u=0}^nr_{A_1,\ldots,A_k}(u)=
\sum_{\substack{0\leq v,w\leq n\\ v+w=n}}r_{A_1,A_2}(v)r_{A_3,\ldots,A_k}(w)\\
\leq&\sum_{v=0}^nr_{A_1,A_2}(v)\sum_{w=0}^nr_{A_3,\ldots,A_k}(w)\leq A_1(n)A_2(n)\sum_{w=0}^nr_{A_3,\ldots,A_k}(w).
\end{align*}
Since $R_{A_1,\ldots,A_k}(n)\gg n$ and $A_1(n),A_2(n)\ll n^{\beta}$, we get that
$$
\sum_{w=0}^nr_{A_3,\ldots,A_k}(w)\geq\frac{R_{A_1,\ldots,A_k}(n)}{A_1(n)A_2(n)}\gg n^{1-2\beta}.
$$
On the other hand, we also have
\begin{align*}
R_{A_1,\ldots,A_k}(3n)=&
\sum_{\substack{0\leq v,w\leq 3n\\ v+w=3n}}r_{A_1,A_2}(v)r_{A_3,\ldots,A_k}(w)\geq
\sum_{v=0}^{2n}r_{A_1,A_2}(v)\sum_{w=0}^{n}r_{A_3,\ldots,A_k}(w)\\
\geq&A_1(n)A_2(n)\sum_{w=0}^nr_{A_3,\ldots,A_k}(w)\gg n^{2\beta}\sum_{w=0}^nr_{A_3,\ldots,A_k}(w).
\end{align*}
It follows from $R_{A_1,\ldots,A_k}(3n)\ll n$  that
$$
\sum_{w=0}^nr_{A_3,\ldots,A_k}(w)\ll n^{1-2\beta}. 
$$
\end{proof}
Assume on the contrary that (\ref{EFAk}) holds. Let $\vartheta(n)=r_{A_1,\ldots,A_k}(n)-cn$.
Let
$$
F_i(z)=\sum_{a\in A_i}z^a
$$
for each $1\leq i\leq k$. Then we have
\begin{align*}
\frac{F_1(z)F_2(z)\cdots F_k(z)}{1-z}=\sum_{n=0}^\infty r_{A_1,\ldots,A_k}(n)z^n
=\frac{cz}{(1-z)^2}+\sum_{n=0}^\infty\vartheta(n)z^n,
\end{align*}
i.e.,
\begin{align*}
&\big(F_1(z)+F_2(z)\big)^2F_3(z)\cdots F_k(z)\\
=&\frac{4cz}{1-z}+4(1-z)\sum_{n=0}^\infty\vartheta(n)z^n+\big(F_1(z)-F_2(z)\big)^2F_3(z)\cdots F_k(z).
\end{align*}
Taking the derivative in $z$, we obtain that
\begin{align*}
&2\big(F_1'(z)+F_2'(z)\big)\big(F_1(z)+F_2(z)\big)\prod_{j=3}^k F_j(z)+
\big(F_1(z)+F_2(z)\big)^2\sum_{j=3}^kF_j'(z)\prod_{\substack{3\leq i\leq k\\ i\neq j}}F_i(z)\\
=&2\big(F_1'(z)-F_2'(z)\big)\big(F_1(z)-F_2(z)\big)\prod_{j=3}^k F_j(z)+
\big(F_1(z)-F_2(z)\big)^2\sum_{j=3}^kF_j'(z)\prod_{\substack{3\leq i\leq k\\ i\neq j}}F_i(z)\\
&+\frac{4c}{(1-z)^2}+4(1-z)\sum_{n=0}^\infty (n+1)\vartheta(n+1)z^n-4\sum_{n=0}^\infty\vartheta(n)z^n.
\end{align*}

Let $\rho=1-1/N$, $z=\rho e^{2\pi\sqrt{-1}\theta}$ and let $m$ be a large integer to be chosen later. 
Let
\begin{align*}
J=&\int_0^1 \bigg|2\big({F_1}'(z)+{F_2}'(z) \big)\prod_{j=3}^k F_j(z)+\big(F_1(z)+F_2(z)\big)\sum_{j=3}^kF_j'(z)\prod_{\substack{3\leq i\leq k\\ i\neq j}}F_i(z)\bigg|\\
 &\cdot\big| F_1(z)+F_2(z)\big|\cdot\bigg|\frac{1-z^m}{1-z} \bigg|^2d\theta,
 \end{align*}
and
\begin{align*}
J_4=&\int_0^1 \bigg|2\big({F_1}'(z)-{F_2}'(z) \big)\prod_{j=3}^k F_j(z)+\big(F_1(z)-F_2(z)\big)\sum_{j=3}^kF_j'(z)\prod_{\substack{3\leq i\leq k\\ i\neq j}}F_i(z)\bigg|\\
 &\cdot\big| F_1(z)-F_2(z)\big|\cdot\bigg|\frac{1-z^m}{1-z} \bigg|^2d\theta.\end{align*}
And let $J_1,J_2,J_3$ be the same ones in (\ref{J1})-(\ref{J3}) respectively.

First, we shall give a lower bound of $J$. Let
$$
G(z):=2\big({F_1}'(z)+{F_2}'(z)\big)\prod_{j=3}^k F_j(z)+\big(F_1(z)+F_2(z)\big)\sum_{j=3}^kF_j'(z)\prod_{\substack{3\leq i\leq k\\ i\neq j}}F_i(z).
$$
Write $$
G(z)=\sum_{n=0}^\infty g_n z^n,\qquad 
\big({F_1}'(z)+{F_2}'(z)\big)\prod_{j=3}^k F_j(z)=\sum_{n=0}^\infty h_n z^n.$$
Clearly $g_n\geq 2h_n\geq 0$ for each $n\geq 0$. Let $\cA$ denote the multiset $A_1\cup A_2$, i.e., 
the common elements of $A_1$ and $A_2$ have the multiplicity $2$ in $\cA$. Then
$$
F_1(z)+F_2(z)=\sum_{a\in A_1}z^a+\sum_{a\in A_2}z^a=\sum_{a\in\cA}z^a.
$$
Thus
\begin{align*}
J\geq&\bigg|\int_0^1 \overline{G(z)}\cdot\big(F_1(z)+F_2(z)\big)\cdot\frac{1-z^m}{1-z}\cdot\frac{1-\overline{z}^m}{1-\overline{z}}d\theta\bigg|\\
=&\bigg|\int_0^1 \bigg(\sum_{n=0}^\infty g_n \overline{z}^n\bigg)\cdot\bigg(\sum_{a\in\cA} z^a\bigg)\cdot 
\bigg(\sum_{n=0}^{m-1} z^n\bigg)\cdot \bigg(\sum_{n=0}^{m-1} \overline{z}^n\bigg)d\theta\bigg|\\
=&\sum_{\substack{a\in\cA,\ u\geq 0\\ 0\leq v,w\leq m-1\\ a+v=u+w}} \rho^{a+u+v+w}g_{u}\geq 
2\sum_{\substack{a\in\cA,\ u\geq 0\\ 0\leq v,w\leq m-1\\ a+v=u+w}} \rho^{a+u+v+w}h_{u}.
\end{align*}

Note that
$$
\big({F_1}'(z)+{F_2}'(z)\big)\prod_{j=3}^k F_j(z)=\sum_{a\in\cA}a z^{a-1}\cdot\sum_{n=0}^\infty r_{A_3,\ldots,A_k}(n) z^n=
\sum_{n=0}^\infty z^{n}\sum_{\substack{a\in\cA\\ a\leq n+1}}ar_{A_3,\ldots,A_k}(n-a+1).
$$
It follows that
\begin{align*}
\sum_{\substack{a\in\cA,\ u\geq 0\\ 0\leq v,w\leq m-1\\ a+v=u+w}} \rho^{a+u+v+w}h_{u}=
&\sum_{\substack{a\in\cA,\ u\geq 0\\ 0\leq v,w\leq m-1\\ a+v=u+w}} \rho^{a+u+v+w}
\sum_{\substack{b\in\cA\\ b\leq u+1}}br_{A_3,\ldots,A_k}(u+1-b)\\
\geq&\sum_{\substack{a,b\in\cA,\ b\leq u\\ 0\leq v,w\leq m-1\\ a+v=u+w}} \rho^{a+u+v+w}
\cdot br_{A_3,\ldots,A_k}(u+1-b).
\end{align*}
We may restrict the above summation to those $a,b,u,v,w$ satisfying the following conditions:\medskip

(1) $c_3N\leq a\leq N$, where $c_3=\big(c_1/(4c_2)\big)^{\frac 1\beta}$;\medskip

(2) $b=a$, $a\leq u\leq N$;\medskip

(3) $1\leq w<m/2$, $w\leq v\leq w+m/2$.\medskip

\noindent Thus
\begin{align*}
J\geq&2\sum_{\substack{a\in\cA,\ c_3N\leq a\leq N\\ 0\leq w<m/2,\ w\leq v\leq w+m/2\\ a\leq u\leq N,\ u-a=v-w}} \rho^{a+u+v+w}
\cdot ar_{A_3,\ldots,A_k}(u+1-a)\\
\geq&2\sum_{\substack{a\in\cA\\ c_3N\leq a\leq N}} a\rho^{a+(w-v+a)+v+w}
\sum_{\substack{0\leq w<m/2\\ w\leq v\leq w+m/2} }r_{A_3,\ldots,A_k}(v-w+1)\\
\geq&2c_3N\cdot\rho^{2N+2m}\big(\cA(N)-\cA(c_3N)\big)\cdot
\frac{m}{2}\sum_{j=0}^{m/2} r_{A_3,\ldots,A_k}(j+1).
\end{align*}
We have $\cA(N)\geq A_1(N)\geq  c_1N^\beta$ and
$$
\cA(c_3N)\leq A_1(c_3N)+A_2(c_3N)\leq 2c_2c_3^\beta N^\beta=\frac{c_1}{2}N^\beta.
$$
Furthermore, since $m\leq N$,
$$
\rho^{2N+2m}\geq\bigg(1-\frac1N\bigg)^{4N}\geq\frac1{2e^4}.
$$
Hence by Lemma \ref{A3AkL},
\begin{align*}
J\gg N^{1+\beta}m\sum_{j=0}^{m/2} r_{A_3,\ldots,A_k}(j+1)\gg m^{2-2\beta}N^{1+\beta}.
\end{align*}

Let us consider the upper bound $J_4$. Clearly
\begin{align*}
J_4\leq &8\int_0^1 \bigg|\frac{F_1'(z)-F_2'(z)}{1-z}\bigg|\cdot\bigg|\frac{F_1(z)-F_2(z)}{1-z}\bigg|\cdot\prod_{j=3}^k|F_j(z)|d\theta\\
 &+4\sum_{j=3}^k\int_0^1\bigg|\frac{F_1(z)-F_2(z)}{1-z}\bigg|^2\cdot|F_j'(z)|\prod_{\substack{3\leq i\leq k\\ i\neq j}}|F_i(z)|d\theta.
 \end{align*}
Note that
\begin{align}
\big|F_3(z)\cdots F_k(z) \big|&\leq \big|F_3(\rho)\cdots F_k(\rho) \big|=\sum_{n=0}^{\infty}\rho^n\cdot r_{A_3,\ldots,A_k}(n).
\end{align}
Let
$$
\omega(x)=\sum_{n\leq x}r_{A_3,\ldots,A_k}(n).
$$
By Lemma \ref{A3AkL}, $\omega(x)\ll x^{1-2\beta}$.
So letting $\eta=-\log\rho$, we have
\begin{align*}
\sum_{n=0}^{\infty}\rho^n\cdot r_{A_3,\ldots,A_k}(n)=&
\int_0^{+\infty}\rho^xd\omega(x)
=\rho^x\omega(x)\big|_0^{+\infty}-\log\rho
\int_0^{+\infty}\omega(x)\cdot \rho^xd x\\
\ll&\eta\int_0^{+\infty} x^{1-2\beta}e^{-\eta x}d x=\frac{\Gamma(2-2\beta)}{\eta^{1-2\beta}}\ll N^{1-2\beta}.
\end{align*}
Thus by the Cauchy-Schwarz inequality,
\begin{align*}
&\int_0^1 \bigg|\frac{F_1'(z)-F_2'(z)}{1-z}\bigg|\cdot\bigg|\frac{F_1(z)-F_2(z)}{1-z}\bigg|\cdot\prod_{j=3}^k|F_j(z)|d\theta\\
\ll&N^{1-2\beta}\bigg(\int_0^1 \bigg|\frac{F_1'(z)-F_2'(z)}{1-z}\bigg|^2d\theta\bigg)^{\frac12}\cdot\bigg(\int_0^1\bigg|\frac{F_1(z)-F_2(z)}{1-z}\bigg|^2d\theta\bigg)^{\frac12}.
\end{align*}
Applying Lemma \ref{NABdiffT}, we obtain that
\begin{align*}
\int_0^1 \bigg|\frac{F'(z)-G'(z)}{1-z}\bigg|^2d\theta\leq&\frac{1}{\rho}\int_0^1 \bigg|\frac{zF'(z)-zG'(z)}{1-z}\bigg|^2d\theta\\
=&\frac1{\rho}\sum_{n=0}^\infty\rho^{2n}\bigg(\sum_{\substack{a_1\in A_1\\ a_1\leq n}}a_1-\sum_{\substack{a_2\in A_2\\ a_2\leq n}}a_2\bigg)^2=o(N^{2+2\beta}),
\end{align*}
and
\begin{align*}
\int_0^1 \bigg|\frac{F(z)-G(z)}{1-z}\bigg|^2d\theta
=\sum_{n=0}^\infty\rho^{2n}\bigg(\sum_{\substack{a_1\in A_1\\ a_1\leq n}}1-\sum_{\substack{a_2\in A_2\\ a_2\leq n}}1\bigg)^2=o(N^{2\beta}).
\end{align*}
So
\begin{align*}
\int_0^1 \bigg|\frac{F_1'(z)-F_2'(z)}{1-z}\bigg|\cdot\bigg|\frac{F_1(z)-F_2(z)}{1-z}\bigg|\cdot\prod_{j=3}^k|F_j(z)|d\theta
\ll N^{1-2\beta}\cdot o(N^{1+2\beta})=o(N^2).
\end{align*}

Similarly, for each $3\leq j\leq k$,
\begin{align*}
|F_j'(z)|\prod_{\substack{3\leq i\leq k\\ i\neq j}}|F_i(z)|\leq&
|F_j'(\rho)|\prod_{\substack{3\leq i\leq k\\ i\neq j}}|F_i(\rho)|=
\sum_{a_3\in A_3,\ldots,a_k\in A_k} a_j\rho^{a_3+\cdots+a_k}\\
=&\sum_{n=0}^\infty  \rho^{n}\sum_{\substack{a_3\in A_3,\ldots,a_k\in A_k\\ a_3+\cdots+a_k=n}} a_j\leq\sum_{n=0}^\infty \rho^{n}\cdot r_{A_3,\ldots,A_k}(n)n.
\end{align*}
Now
\begin{align*}
&\sum_{n=0}^{\infty}\rho^nn\cdot r_{A_3,\ldots,A_k}(n)=
\int_0^{+\infty}\rho^xxd\omega(x)\\
=&\rho^xx\cdot\omega(x)\big|_0^{+\infty}-
\int_0^{+\infty}\omega(x)\cdot \rho^x(1+x\log\rho)d x\\
\ll&\eta\int_0^{+\infty} x^{1-2\beta}e^{-\eta x}d x+\int_0^{+\infty} x^{1-2\beta}e^{-\eta x}d x\\
=&\frac{\Gamma(3-2\beta)+\Gamma(2-2\beta)}{\eta^{2-2\beta}}\ll N^{2-2\beta}.
\end{align*}
It follows that
\begin{align*}
\int_0^1 \bigg|\frac{F(z)-G(z)}{1-z}\bigg|^2\cdot |F_j'(z)|\prod_{\substack{3\leq i\leq k\\ i\neq j}}|F_i(z)|d\theta
\ll&N^{2-2\beta}\int_0^1 \bigg|\frac{F(z)-G(z)}{1-z}\bigg|^2d\theta\\
=&
N^{2-2\beta}\cdot o(N^{2\beta})=o(N^2).
\end{align*}
Thus we get
$$
J_4=o(N^2).
$$

Recall that $J\leq J_1+J_2+J_3+J_4$ and 
$$
J_1,J_2\ll m^2N,\qquad J_3=o(m^{\frac12}N^{\frac 74}).
$$
We may choose a large constant $C>1$ such that
$$
m N^{\frac32}\leq Cm^2N+o(m^{\frac12}N^{\frac 74})+o(N^2).
$$
It immediately leads to a contradiction by setting $m=C^{-2}N^{\frac12}$.

\section{Proof of Theorem \ref{mainT3}}
\setcounter{lemma}{0}\setcounter{theorem}{0}\setcounter{proposition}{0}\setcounter{corollary}{0}
\setcounter{equation}{0}

Here we only give the proof of (i) of Theorem \ref{mainT3}, since the proof of (ii) is completely same.

Suppose that $N$ is sufficiently large and $\rho=1-1/N$. Let $\vartheta(n)=R_{A,B}(n)-cn$ and let $J,J_1,J_2,J_3,J_4$  be given by (\ref{J})-(\ref{J4}).
We shall give the upper bounds of $J_2,J_3$ under the assumption
\begin{equation}
\varpi(x):=\sum_{n\leq x}\vartheta(n)^2=o(x^{\frac{3}{2}}),\qquad x\to+\infty.
\end{equation}

By the Cauchy-Schwarz inequality,
\begin{align*}
J_2=&\int_0^1 \bigg|\sum_{n=0}^\infty \vartheta(n)z^n \bigg|\cdot\bigg|\frac{1-z^m}{1-z} \bigg|^2d\theta\\
\leq&\bigg(\int_0^1 \bigg|\sum_{n=0}^\infty \vartheta(n)z^n \bigg|^2d\theta\bigg)^{\frac12}\cdot\bigg(\int_0^1\bigg|\frac{1-z^m}{1-z} \bigg|^4d\theta\bigg)^{\frac12},
\end{align*}
where $z=\rho e^{2\pi\sqrt{-1}\theta}$.
Clearly
$$
\int_0^1\bigg|\frac{1-z^m}{1-z} \bigg|^4d\theta\leq\sum_{\substack{0\leq a,b,c,d\leq m-1\\ a+b=c+d}}1\leq m^3.
$$
And
$$
\int_0^1 \bigg|\sum_{n=0}^\infty \vartheta(n)z^n \bigg|^2d\theta
=\sum_{n=0}^\infty \vartheta(n)^2\rho^{2n}.
$$
Since $\varpi(x)=o(x^{\frac32})$, for any $\epsilon>0$, there exists $x_0=x_0(\epsilon)>0$ such that for any $x\geq x_0$, $\varpi(x)\leq \epsilon x^{\frac32}$. Note that trivially $\varpi(x)\leq x^5$ for any $x\geq 0$. Letting $\eta=-\log\rho$, we have
\begin{align*}
\sum_{n=0}^\infty \vartheta(n)^2\rho^{2n}=\int_0^{+\infty}\rho^{2x}d\varpi(x)
=&\varpi(x)\rho^{2x}\big|_0^{+\infty}-2\log\rho\int_0^{+\infty}\varpi(x)\rho^{2x}d x\\
\leq &2\eta\int_0^{+\infty}\epsilon x^{\frac 32}\cdot e^{2\eta x}d x+2\eta\int_0^{x_0}x^5d x\\
\leq&
\frac{\epsilon\cdot\Gamma(\frac 52)}{(2\eta)^{\frac 32}}+\frac{\eta x_0^6}{3}\leq2\epsilon N^{\frac32},
\end{align*}
provided that $N\geq\epsilon^{-1}x_0^6$.
So$$
\sum_{n=0}^\infty \vartheta(n)^2\rho^{2n}=o(N^{\frac32})
$$
as $N\to+\infty$, and
$$
J_2=o(m^{\frac32}N^{\frac 34}).
$$

On the other hand, since $|1-z^m|\leq 2$, we have
\begin{align*}
J_3=&\int_0^1 \bigg|(1-z)\sum_{n=0}^\infty (n+1)\vartheta(n+1)z^n \bigg|\cdot\bigg|\frac{1-z^m}{1-z} \bigg|^2d\theta\\
\leq&2\int_0^1 \bigg|\sum_{n=0}^\infty (n+1)\vartheta(n+1)z^n \bigg|\cdot\bigg|\frac{1-z^m}{1-z} \bigg|d\theta\\
\leq&2\bigg(\int_0^1 \bigg|\sum_{n=0}^\infty (n+1)\vartheta(n+1)z^n\bigg|^2d\theta\bigg)^{\frac12}\cdot\bigg(\int_0^1 \bigg|\frac{1-z^m}{1-z}\bigg|^2d\theta\bigg)^{\frac12}.
\end{align*}
Clearly
$$
\int_0^1 \bigg|\frac{1-z^m}{1-z}\bigg|^2d\theta\leq m.
$$
And
\begin{align*}
\int_0^1 \bigg|\sum_{n=0}^\infty (n+1)\vartheta(n+1)z^n\bigg|^2d\theta=
\sum_{n=0}^\infty (n+1)^2\vartheta(n+1)^2\rho^{2n}\leq 
4\sum_{n=0}^\infty n^2\vartheta(n)^2\rho^{2n}.
\end{align*}
Note that $(x^2\rho^{2x})'=2x\rho^{2x}+2\log\rho\cdot x^2\rho^{2x}$. For any $\epsilon>0$, we have
\begin{align*}
\sum_{n=0}^\infty n^2\vartheta(n)^2\rho^{2n}=&\int_0^{+\infty} x^2\rho^{2x}d\varpi(x)=-2\int_0^\infty(x\rho^{2x}+\log\rho\cdot x^2\rho^{2x})\cdot\varpi(x)d x\\
\leq&2\eta\int_0^{+\infty}\epsilon x^{\frac 32}\cdot x^2e^{2\eta x}d x+
2\int_0^{+\infty}\epsilon x^{\frac 32}\cdot xe^{2\eta x} d x+2\int_0^{x_0}(x+\eta\cdot x^2)x^5d x\\
=&\frac{\epsilon\cdot(\Gamma(\frac 92)+2\Gamma(\frac 72))}{(2\eta)^{\frac 72}}+\frac{2x_0^{7}}{7}+\frac{\eta x_0^{8}}{4}\leq 4N^{\frac 72},
\end{align*}
provided that $N\geq \epsilon^{-1} x_0^8$.
Hence
$$
J_3=o(m^{\frac12}N^{\frac74}).
$$

Finally, for any $\epsilon>0$, clearly
$$
\sum_{n\leq N}\vartheta(n)^2\leq \epsilon N^{\frac{3}{2}}
$$
implies  $\vartheta(n)\leq \epsilon^{\frac12} N^{\frac{3}{4}}$. Thus we have
$$
R_{A,B}(n)=cn+o(n^{\frac34})
$$
as $n\to+\infty$. Since $R_{A,B}(n)=\Psi(n)$, under the assumptions of Theorem \ref{mainT1}, we know that
$J\gg m N^{\frac 32}$, $J_4=o(N^2)$ and $J_1=O(m^2N)$. It follows from $J\leq J_1+J_2+J_3+J_4$ that
$$
m N^{\frac 32}\leq Cm^2N+o(m^{\frac32}N^{\frac 34})+o(m^{\frac12}N^{\frac74})+o(m^{\frac12}N^{\frac74})
$$
for some constant $C>1$. Setting $m=C^{-2}N^{\frac12}$, we immediately get a contradiction whenever $N$ is sufficiently large.\qed


\begin{thebibliography}{99}

\bibitem {Ba77} P. T. Bateman, {\it The Erd\H os-Fuchs theorem on the square of a power series}, J. Number Theory, {\bf 9}(1977),  330-337.

\bibitem{CT11} Y.-G. Chen and M. Tang, \textit{A quantitative Erd\H os-Fuchs theorem and its generalization}, Acta Arith., \textbf{149}(2011), 171-180.

\bibitem{CT12} Y.-G. Chen and M. Tang, \textit{A generalization of the classical circle problem}, Acta Arith., {\bf 152}(2012), 279-290.

\bibitem{EF56} P. Erd\H os and W. J. Fuchs, \textit{On a problem of additive number theory}, J. London Math. Soc., \textbf{31}(1956), 67-73.

\bibitem{Ho01} G. Horv\'ath, \textit{On a generalization of a theorem of Erd\H os and Fuchs}, Acta Math. Hungar., \textbf{92}(2001), 83-110.

\bibitem{Ho02} G. Horv\'ath, \textit{On a theorem of Erd\H os and Fuchs}, Acta Arith., \textbf{103}(2002), 321-328.

\bibitem{Ho04} G. Horv\'ath, \textit{An improvement of an extension of a theorem of Erd\H os and Fuchs}, Acta Math. Hungar., \textbf{104}(2004), 27-37.

\bibitem{Hu02} M. N. Huxley, {\it Integer points, exponential sums and the Riemann zeta function, in: Number Theory for the Millennium, II}, Urbana, IL, 2000, A K Peters, Natick, MA, 2002, pp. 275��90.

\bibitem{MV90} H. L. Montgomery and R. C. Vaughan, \textit{On the Erd\H os-Fuchs theorems}, A tribute to Paul Erd\H os, 331-338, Cambridge Univ. Press, Cambridge, 1990.


\bibitem{Sa80} A. Sark\"ozy, {\it On a theorem of Erd\H os and Fuchs}, Acta Arith., {\bf 37}(1980), 333-338.

\bibitem{Ta09} M. Tang, \textit{On a generalization of a theorem of Erd\H os and Fuchs}, Discrete Math., \textbf{309}(2009), 6288-6293.

\bibitem{Ta14} M. Tang, {\it On a generalization of the Erd\H os-Fuchs theorem}, Int. J. Number Theory, {\bf 10}(2014), 955-961.



\end{thebibliography}
\end{document}